\documentclass[12pt]{article}
\usepackage{amssymb}
\usepackage{graphicx}

\usepackage{color}
\usepackage{latexsym}
\usepackage{amssymb,amsmath,amstext}
\usepackage{epsfig}
\usepackage{graphics}
\usepackage{subfigure}
\usepackage{euscript}
\usepackage{ifthen}
\usepackage{wrapfig}
\usepackage{amsthm}
\usepackage{url}

\theoremstyle{plain}%
\newtheorem{theorem}{Theorem}
\newtheorem{proposition}[theorem]{Proposition}
\newtheorem{example}[theorem]{Example}

\newtheorem{corollary}[theorem]{Corollary}

\newtheorem{remark}[theorem]{Remark}

\newcommand{\PP}{\mathbb{P}}

\date{}
\begin{document}

\title{Geometry of Higher-Order Markov Chains}

\author{Bernd Sturmfels\footnote{{\em Department~of Mathematics,
University of California at Berkeley, Berkeley, CA 94720, USA,}
{\tt bernd@math.berkeley.edu}. \ \ \
This research project was supported in part by the
 National Science Foundation
(DMS-0968882) and the DARPA Deep Learning program (FA8650-10-C-7020).}
}

\maketitle

 \begin{abstract}
 \noindent
We determine an explicit Gr\"obner basis,
consisting of linear forms and determinantal quadrics,
for the prime ideal of Raftery's mixture transition distribution model for Markov chains.
When the states are binary, the corresponding projective 
variety is a linear space, the model  itself consists
of two simplices in a cross-polytope, and the likelihood function
typically has two local maxima.
 In the general non-binary case, the model
corresponds to a cone over a Segre variety.
\end{abstract}

\section{Introduction}

In this note we investigate Adrian Raftery's
{\em mixture transition distribution model} (MTD)
from the perspective of algebraic statistics \cite{LAS, ascb}.
The MTD model, which was first proposed in \cite{Raf},
has a wide range of applications in engineering and the sciences \cite{RT}.
The article by Berchtold and Raftery \cite{BR} offers a detailed introduction and review.

The point of departure for this project was a conjecture due to
Donald Richards \cite{Ric}, stating that the
likelihood function of an MTD model can have multiple local maxima.
We establish this conjecture for the case of binary states
in Proposition \ref{thm:twolocal}. 

Our main result, to be
derived in Section 4, gives an explicit Gr\"obner basis
for the MTD model. Here, both the sequence length and the number of states are arbitrary.

We begin with an algebraic description of the model in \cite{BR, Raf}.
Fix a pair of positive integers $l$ and $m$,
and set $N = m^{l+1}-1$. We define
the statistical model ${\rm MTD}_{l,m}$ whose state space is the set
$[m]^{l+1}$ of sequences $i_0 i_1 \cdots i_l$ of length $l+1$
over the alphabet $[m] = \{1,2,\ldots,m\}$.
The model has $(m-1)m + l-1$ parameters, given
by the entries of an $m \times m$-transition matrix  $(q_{ij})$
and a probability distribution $\lambda = (\lambda_1,\ldots,\lambda_l)$ on
the set $[l] = \{1,2,\ldots,l\}$ of the hidden states.  Thus the parameter space is
the product of simplices $\,(\Delta_{m-1})^m \times \Delta_{l-1}$. 
The model ${\rm MTD}_{l,m}$ will be a semialgebraic subset
of the simplex $\Delta_N$. That simplex has its coordinates
$p_{i_0 i_1 \cdots i_l}$ indexed by sequences in $[m]^{l+1}$.

The model ${\rm MTD}_{l,m}$ 
 is the image of the bilinear map
$$ \phi_{l,m} \,: \, \,(\Delta_{m-1})^m \times \Delta_{l-1} \,
\rightarrow \, \Delta_N $$
which is defined by the formula
\begin{equation}
\label{eq:param} p_{i_0 i_1 \ldots i_{l-1} i_l} \quad = \quad
\frac{1}{m^{l}} \cdot \sum_{j=1}^{l} \lambda_j q_{i_{j-1},i_l} 
\end{equation}
As is customary in algebraic statistics, we
pass to a simpler object of study
by considering the Zariski closure
$\overline{{\rm MTD}}_{l,m}$ of our model in the
complex projective space $\PP^N$, and we seek to
compute the homogeneous prime ideal of all polynomials
in the $N+1$ unknowns $p_{i_0 i_1 \ldots i_l}$ that vanish on
$\overline{{\rm MTD}}_{l,m}$.
This  particular goal will be reached in our Theorem \ref{thm:main}.

The following probabilistic interpretation of 
the formula $(\ref{eq:param})$ makes it evident that 
$\sum p_{i_0 i_1 \cdots i_l} = 1$ holds on the image of $\phi_{l,m}$.
We generate a sequence of length $l+1$ on $m$ states as follows.
First we select from the uniform distribution on all
$m^l$ sequences $i_0 i_1 \cdots i_{l-1}$ of length $l$.
All that remains is to  determine the
state  $i_l$ in position $l$. The mixture distribution $\lambda$ 
determines which of the earlier states
gets used in the transition. With probability 
$\lambda_j$, we select position $j-1$ for that.
The character in the last position $l$ is determined
from the state $i_{j-1}$ in position $j-1$ using the transition matrix $(q_{ij})$.

The model ${\rm MTD}_{l,m}$ is known to be identifiable \cite[\S 4.2]{BR}.
Consequently, the dimension of the projective variety $\overline{{\rm MTD}}_{l,m}$ 
is equal to the number $(m-1)m + l-1$ of model parameters.
A geometric characterization of this variety 
will be given in Corollary \ref{cor:segre}.

Equations defining Markov chains and Hidden Markov Models
have received considerable attention in algebraic statistics \cite{Cri, HT, HDY, Sch}.
We contribute to this literature by studying the algebraic geometry
 of a fundamental model for higher order Markov chains.
In addition to our theoretical results in
Theorems \ref{thm:eins} and \ref{thm:main},
readers from statistics will find in Section 3 an analysis
of the behavior of the EM algorithm for binary MTD models.

\section{Binary States}

Our first result concerns the geometry of the model 
in the case $m=2$ of binary states.

\begin{theorem}
\label{thm:eins}
The variety $\,\overline{{\rm MTD}}_{l,2}\,$
is a linear subspace of dimension $l+1$ in the projective space $\PP^N$.
This variety intersects the probability simplex $\Delta_N$ in a regular cross-polytope of
dimension $l+1$.
The model ${\rm MTD}_{l,2}$ is the union of two $(l+1)$-simplices
spanned by vertices of the cross-polytope  $\overline{{\rm MTD}}_{l,2} \cap \Delta_N$.
The two simplices meet along a common~edge.
\end{theorem}

The {\em cross-polytope} is the free object in the category of centrally symmetric 
polytopes \cite{Zie}. It can be represented  as the convex hull
of all signed unit vectors $e_i$ and $-e_i$ where $i=0,1,\ldots,l$, so it is
an $(l+1)$-dimensional polytope with $2l+2$ vertices and $2^{l+1}$ facets.

Before we come to the proof Theorem \ref{thm:eins}, let us first see some examples to illustrate it. In what follows 
we abbreviate the model parameters by $q_{11} = a$, $q_{21} = b$ and $\lambda_2 = \lambda$.

\begin{example} \rm
Theorem \ref{thm:eins} also 
applies in the trivial case $l = 1$, where 
(\ref{eq:param}) reads
\begin{equation}
\label{eq:2by2}  \begin{pmatrix} p_{11} & p_{12} \\
      p_{21} & p_{22} \end{pmatrix}
      \quad = \quad
      \begin{pmatrix} a/2 & (1-a)/2 \\ b/2 & (1-b)/2 \end{pmatrix}. 
\end{equation}
      The variety $\overline{{\rm MTD}}_{1,2}$
is the plane in $\PP^3$ given by $p_{11} + p_{12} = p_{21} + p_{22}$.
Its intersection with the tetrahedron $\Delta_3$ coincides with the model 
${\rm MTD}_{1,2}$,
which is a regular square:
$$ 
{\rm MTD}_{1,2} \,\, = \,\,
\overline{{\rm MTD}}_{1,2} \cap \Delta_3 \,= \,
{\rm conv} \biggl\{ 
\begin{pmatrix} 1/2 &\! 0 \\ 1/2 &\! 0 \end{pmatrix},\,
\begin{pmatrix} 1/2 &\! 0 \\ 0 &\!\!1/2 \end{pmatrix},\,
\begin{pmatrix} 0  &\!\! 1/2  \\ 1/2 &\! 0 \end{pmatrix},\,
\begin{pmatrix} 0 &\! 1/2 \\ 0 &\! 1/2 \end{pmatrix}
\biggr\}.
$$
The first three and last three matrices in this list form the
two triangles referred to in Theorem \ref{thm:eins}.
Their common edge consists of all transition matrices (\ref{eq:2by2}) of rank~$1$.~\qed
\end{example}

\begin{example} \rm
Our first non-trivial example arises for $l=m=2$.
The map $\phi_{2,2}$ is given~by
$$
\!\! (a,b,\lambda) \mapsto
p = \frac{1}{4} \! \begin{bmatrix}
a e_{111} +
(\lambda b  + (1-\lambda) a )  e_{121} + 
(\lambda a + (1-\lambda)  b )  e_{211} + 
b e_{221} + 
(1{-}a) e_{112} \\ + 
(\lambda (1{-}b)  {+} (1{-}\lambda) (1{-}a) ) e_{122} + 
(\lambda (1{-}a) {+} (1{-}\lambda) (1{-}b )) e_{212} + (1{-}b) e_{222} 
\end{bmatrix}
$$
Here $\{e_{111}, e_{112},\ldots, e_{222}\}$ denotes the standard basis in the space
of $2 \times 2 \times 2$-tensors.
The variety $\overline{{\rm MTD}}_{2,2} $
is the $3$-dimensional linear subspace of $\PP^7$ defined by
$$
\begin{matrix}
p_{111}+p_{112} = p_{121} + p_{122}, &
p_{211}+p_{212} = p_{221} + p_{222} , \\
p_{121}+p_{122} = p_{221} + p_{222}, &
p_{111}+p_{221} = p_{121} + p_{211}.
\end{matrix}
$$
The intersection of this linear space with the simplex  $\Delta_7$
is the regular octahedron whose vertices are the images  under $\phi_{2,2}$
of the vertices of the cube $\,(\Delta_1)^2  \times \Delta_1 $.
The model ${\rm MTD}_{2,2}$ consists of two tetrahedra formed by vertices
of the octahedron. Their common edge is the segment between
$\frac{1}{4}(e_{111}+e_{121}+e_{211}+e_{221})$ and 
$\frac{1}{4}(e_{112}+e_{122}+e_{212}+e_{222})$. \qed
\end{example}

\begin{example}
\label{ex23}
 \rm
The statement of Theorem \ref{thm:eins} does not extend to $m \geq 3$.
Consider the case $l=2, m = 3$. The $7$-dimensional variety
$\overline{{\rm MTD}}_{2,3} $ lives in $\PP^{26}$,
and it is not a linear space.
The linear span of  $\overline{{\rm MTD}}_{2,3} $  is $10$-dimensional. 
Inside this $\PP^{10}$, the variety $\,\overline{{\rm MTD}}_{2,3} $
has codimension $3$, degree $4$, and  it is cut out by six quadrics.
In Example~\ref{ex:neun} we shall
display a Gr\"obner basis  consisting of $16$ linear forms
and six quadrics for its prime ideal.
\qed 
\end{example}

\begin{proof}[Proof of Theorem \ref{thm:eins}]
It is known by \cite[\S 4.2]{BR} that the model is identifiable, so 
${\rm MTD}_{l,2}$ is a semi-algebraic set of dimension $l+1$ in $\Delta_N$.
Its Zariski closure  $\overline{{\rm MTD}_{l,2}}$ is a
variety of dimension $l+1$ in $\PP^N$. That variety is irreducible because it is defined by
way of  a rational parametrization.
For any binary sequence $i_0 i_1 \cdots i_{l-1}$, the identity
\begin{equation}
\label{eq:lin1}
p_{i_0 i_1 \cdots i_{l-1} 2} \,\,=\,\,  2^{-l} -  p_{i_0 i_1 \cdots i_{l-1} 1} 
\end{equation}
holds on  ${\rm MTD}_{l,2}$, so it suffices to consider 
relations on probabilities of sequences that end with $1$.
On our model, these probabilities satisfy the linear equations
\begin{equation}
\label{eq:lin2}
p_{i_0 i_1 \cdots i_r \cdots i_s \cdots  i_{l-1} 1} + 
p_{i_0 i_1 \cdots \tilde{i}_r \cdots \tilde{i}_s \cdots  i_{l-1} 1} 
\,\,=\,\,
p_{i_0 i_1 \cdots i_r \cdots \tilde{i}_s \cdots  i_{l-1} 1} + 
p_{i_0 i_1 \cdots \tilde{i}_r \cdots i_s \cdots  i_{l-1} 1} . 
\end{equation}
In other words, the $l$-dimensional $2 {\times} 2 {\times} \cdots {\times} 2$-tensor 
$(p_{i_0 i_1 \cdots i_{l-1} 1})$ has tropical rank $1$.
The set of such tensors is a classical linear space of dimension $l+1$.

Solving the linear equations  (\ref{eq:lin1}) and (\ref{eq:lin2})
on the simplex $\Delta_N$, we obtain
an $(l+1)$-dimensional polytope $P$ that contains the model ${\rm MTD}_{l,2}$.
Its Zariski closure in $\PP^N$ is an $(l+1)$-dimensional linear space
that contains the variety $\overline{{\rm MTD}_{l,2}}$. 
Being irreducible varieties of the same dimension, they must be equal.
This proves the first assertion.

We next claim that the polytope $P$ of all non-negative real solutions
to (\ref{eq:lin1}) and (\ref{eq:lin2}) is a regular cross-polytope. 
 For $r \in \{0,1,\ldots,l-1\}$ and $s \in \{1,2\}$ define the $2l$ points
$$ E_{rs} \,\,\,=\,\,\, \frac{1}{2^l} \cdot
\biggl[ \sum \bigl\{ \,e_{i_0 i_1 \cdots i_{l-1} 1} \,| \, i_r = s \,\bigr\} \,+\,  \sum \bigl\{ \,e_{i_0 i_1 \cdots i_{l-1} 2} \,| \, i_r \not= s \,\bigr\} \biggr]  \quad \in \,\,\Delta_N . $$
These are extreme non-negative solutions of (\ref{eq:lin1}) and (\ref{eq:lin2}).
They form the vertices of an $l$-dimensional cross-polytope, since
$\,\frac{1}{2}(E_{r1} + E_{r2}) \,$ is equal to the uniform distribution $\, \frac{1}{2^{l+1}} e_{++\cdots++}\,$ for all $r$.
In addition to the $2l$ vertices $E_{rs}$, the polytope $P$ has two more vertices, namely,
$\,\frac{1}{2^l} e_{++\cdots+1}\,$ and $\,\frac{1}{2^l} e_{++\cdots+2}$.
Hence $P$ is a bipyramid over the $l$-dimensional cross-polytope, so 
it is an $(l+1)$-dimensional cross-polytope.

It remains to identify the model ${\rm MTD}_{l,2}$ inside $P$.
The parameter polytope is the product $(\Delta_1)^2 \times \Delta_{l-1}$, and,
as before, we chose coordinates $(a,b)$ on the square $(\Delta_1)^2$.
The map $\phi_{l,2}$ contracts the simplex $\{(0,0)\} \times \Delta_{l-1}$
onto the vertex  $\,\frac{1}{2^l} e_{++\cdots+2}$ of $P$, and it contracts
the simplex $\{(1,1)\} \times \Delta_{l-1}$ onto the vertex $\,\frac{1}{2^l} e_{++\cdots+1}$ of $P$.
The vertex $(0,1) \times e_r$ is mapped to the vertex $E_{r,2}$,
and the vertex $(1,0) \times e_r$ is mapped to the vertex $E_{r,1}$.
The parameter points with $a = b$ are contracted onto the line segment
$S  = [\frac{1}{2^l} e_{++\cdots+1},\frac{1}{2^l} e_{++\cdots+2}]$.
The parameter points with $a < b$ are mapped bijectively onto 
the $(l+1)$-simplex formed by $S$ and $\{E_{0,2}, E_{1,2}, \ldots, E_{l-1,2}\}$, but with $S$ removed.
The parameter points with $a > b$ are mapped bijectively onto 
the $(l+1)$-simplex formed by $S$ and $\{E_{0,1}, E_{1,1}, \ldots, E_{l-1,1}\}$, but with $S$ removed.
Hence ${\rm MTD}_{l,2}$ equals the union of two $(l+1)$-simplices
glued along the special diagonal $S$ of the cross-polytope $P$.
\end{proof}

\begin{corollary}
For large $l$, there are far fewer distributions in the model
${\rm MTD}_{l,2}$
than  distributions in its Zariski closure.
Namely, with respect to Lebesgue measure, we have
$$ \frac{{\rm vol}({\rm MTD}_{l,2})}{
{\rm vol}(\overline{{\rm MTD}}_{l,2} \cap \Delta_N)}
\,\, = \,\,\frac{1}{2^{l-1}}. $$
\end{corollary}

\begin{proof}
We can triangulate the cross-polytope  $P$ into $2^l$ simplices,
all of the same volume and containing the special diagonal $S$. 
The  model ${\rm MTD}_{l,2}$ consists of two of them.
Hence $2/2^l$ is the fraction of the
volume of 
$P = \overline{{\rm MTD}}_{l,2} \cap \Delta_N$ that is occupied by ${\rm MTD}_{l,2}$.
\end{proof}

\section{Likelihood inference}

We next discuss maximum likelihood estimation (MLE) for the
mixture transition distribution model ${\rm MTD}_{l,m}$.
Any data set is represented by a function $\,u : [m]^{l+1} \rightarrow \mathbb{N}\,$
that records the frequency counts of the observed sequences. Given such a function $u$,
our objective is to maximize the corresponding log-likelihood function
\begin{equation}
\label{eq:loglike}
L_u \quad = \quad \sum_{i_0 i_1 \cdots i_l} u_{i_0 i_1 \cdots i_l} \cdot {\rm log}(p_{i_0 i_1 \cdots i_l}) 
\end{equation}
over all probability distributions that lie in the model ${\rm MTD}_{l,m}$.
A standard method for solving this optimization problem is the 
expectation-maximization (EM) algorithm.
Other algorithms for the same task can be found in \cite{Ber, RT}.

A general version of the EM algorithm for algebraic models with discrete data is described in \cite[\S 1.3]{ascb}, 
while the specific case of the MTD model is treated in \cite[\S 4.5]{BR}.
Richards \cite{Ric} conjectured that the EM algorithm for the MTD model
may get stuck in local maxima. Our next result confirms that this is indeed
the case, even for $m=2$.

\begin{proposition}
\label{thm:twolocal}
The log-likelihood function  $L_u$ on the binary model ${\rm MTD}_{l,2}$
has either one or two local maxima.
With probability one, there will be two local maxima, and
both of these will be reached by the EM algorithm
for different  choices of initial parameters.
\end{proposition}

Here the statement about ``probability one'' in the second sentence refers to any
absolutely continuous probability distribution that is positive on the simplex $\Delta_N$.

\begin{proof}
We saw in Theorem \ref{thm:eins} that ${\rm MTD}_{l,2}$
is the union of two convex polytopes. The log-likelihood function $L_u$
is strictly concave on the ambient simplex $\Delta_N$, so it attains a unique maximum on
each of the two polytopes. This proves the first statement.

For the second statement consider the empirical distribution
$u/|u|$ which is a point in $\Delta_N$. Its log-likelihood function
$L_u$ has a unique maximum $p^*$ in the interior of
the cross-polytope $P$. With probability one, this 
maximum $p^*$ will not lie in the segment $S$, so let
us assume that this is the case. Then
either $p^*$ lies in precisely one of the two $(l+1)$-simplices
that make up ${\rm MTD}_{l,2}$, or $p^*$ does not lie in
${\rm MTD}_{l,2}$. In the former case, $p^*$ is the MLE,
and the maximum over the other simplex is in the boundary of
that simplex and constitutes a second local maximum.
In the latter case, each of the two simplices has a local maximum
in its boundary. When choosing starting parameter values near either of these
local maxima, the EM algorithm converges to that local maximum.
\end{proof}

The point $p^*$ in the cross-polytope $P$ at which $L_u$ attains its maximum
is an algebraic function of the data $u$. The degree of this algebraic function
is the  {\em ML degree} (see \cite{HKS})
of the linear subvariety $\,\overline{{\rm MTD}}_{l,2}\,$ of $\PP^N$.
By Varchenko's Formula  \cite[Theorem 1.5]{ascb}, this ML degree
coincides with the number of bounded regions in an arrangement of hyperplanes.
This arrangement lives   inside the affine space that is
cut out by (\ref{eq:lin1}) and (\ref{eq:lin2})
and it consists of the restrictions of the coordinate hyperplanes  $\{ p_\bullet = 0\}$.

 Computations show that the ML degree equals
$9$ for $l = 3$, and it equals $209$ for $l = 4$.
It would be interesting to find a general formula for that ML degree
as a function of $l$.

The local maxima that occur on the boundary of the two simplices
of ${\rm MTD}_{l,2}$ have ML degree $1$, that is, they are
expressed as rational functions in the data $u$. Indeed,
these local maxima are precisely the estimates for the 
Markov chain obtained by fixing
$\lambda_i = 1$ for some $i$. Hence, if $p^* \not\in {\rm MTD}_{l,2}$,
then the MLE is a rational expression in  $u$.
The next example illustrates the behavior of the EM algorithm
for  $m=2$ and $l = 3$.

\begin{example} \rm
The data consists of eight positive integers, here written as a matrix
$$ U \,\,\, = \,\,\,
\begin{pmatrix}
u_{111} & u_{121} & u_{211} & u_{221} \\
u_{112} & u_{122} & u_{212} & u_{222} 
\end{pmatrix}.
$$
The MLE $\hat p$ will be either 
$$
p' \quad = \quad
\frac{1}{2|u|}
\begin{pmatrix}
u_{111} + u_{211} & u_{121} + u_{221}  & u_{111} + u_{211} & u_{121} + u_{221}  \\
u_{112} + u_{212}  & u_{122} +  u_{222}  & u_{112} + u_{212}  & u_{122} +  u_{222}  
\end{pmatrix}
$$
or
$$
p'' \quad = \quad
\frac{1}{2|u|}
\begin{pmatrix}
u_{111}+u_{121} & u_{111}+u_{121} & u_{211}+u_{221} & u_{211}+u_{221} \\
u_{112}+u_{122} & u_{112}+u_{122} & u_{212}+u_{222} & u_{212}+u_{222} & 
\end{pmatrix},
$$
or it will be the unique probability distribution  satisfying
(\ref{eq:lin1}), (\ref{eq:lin2}), and
\begin{equation}
\label{eq:rank4}
 {\rm rank}  \begin{pmatrix}
u_{111} & u_{112} & u_{121} & u_{122} & u_{211} & u_{212} & u_{221} & u_{222} \\
p_{111} & p_{112} & p_{121} & p_{122} & p_{211} & p_{212} & p_{221} & p_{222} \\
p_{111} & p_{112} & -p_{121} & -p_{122} & 0 &    0  &  0  &   0     \\
  0  &  0  &  0  &   0   &  p_{211} &  p_{212} & -p_{221} &  -p_{222} \\
  0  &  0  &  p_{121} &  p_{122} &  0  &   0 &   -p_{221} & -p_{222} \\
p_{111} &  0  &  -p_{121} & 0  &   -p_{211} & 0  &  p_{221} &   0     
\end{pmatrix} \,\leq \,5 .
\end{equation}
This is the matrix denoted  $\begin{bmatrix} u \\ \tilde{J} \end{bmatrix} $
in \cite[\S 3]{HKS}. The rank constraint (\ref{eq:rank4})
represents Proposition~2 in \cite{HKS}.
The unique probability distribution that lies in our model and also satisfies
  (\ref{eq:rank4}) was  called  $p^*$ in the proof of
Proposition \ref{thm:twolocal}. Its defining constraints 
(\ref{eq:lin1}), (\ref{eq:lin2}) and (\ref{eq:rank4}) form a system of polynomial equations
that has $9$ complex solutions. The distribution $p^*$ is the unique solution
to that system whose coordinates are both real and positive.

The trichotomy in this example is best explained
by the following observations:
For almost all data matrices $U$,  the three points $p', p'', p^*$ are distinct,
one of them coincides with the global maximum $\hat p$ of $L_u$
over ${\rm MTD}_{l,2}$, and another one is a local maximum.
 \qed
\end{example}

It would be interesting to extend the findings in Proposition \ref{thm:twolocal}
to $m \geq 3$. The algebraic tools that may be needed for such an 
analysis are developed in the next section.

\section{Non-linear Models}

In this section we examine the geometry of model ${\rm MTD}_{l,m}$
and the variety $\overline{{\rm MTD}}_{l,m}$ for an arbitrary number $m$
of states. In particular, we prove that its prime ideal 
is minimally generated by linear forms and quadrics. These minimal
generators form a Gr\"obner basis.

\begin{theorem} \label{thm:main}
The variety $\overline{{\rm MTD}}_{l,m}$ spans a linear space
of dimension  $(m-1)(lm-l+1)$ in $\PP^N$. 
In this linear space, its prime ideal is given by
the $2 \times 2$-minors of an $l \times (m-1)^2$-matrix of linear forms.
The linear and quadratic ideal generators form a Gr\"obner basis.
\end{theorem}

This theorem explains our earlier result that the model is linear
for binary states. Indeed, for $m=2$,  the dimension $(m-1)m+l-1$ of the model
coincides with the dimension $ (m-1)(lm-l+1)$ of the ambient linear space, and
there are no $2 \times 2$-minors.

\begin{proof}
We shall present an explicit Gr\"obner basis consisting of linear forms and quadrics.
The term order we choose is the reverse lexicographic term order
induced by the lexicographic order on the states $i_0 i_1 \cdots i_l $ of the model.
We first consider the linear relations
\begin{equation}
\label{eq:linrel1} \underline{ p_{i_0 i_1 i_2 \cdots i_{l-1} i_l}}
- \sum_{j=0}^{l-1} p_{m \cdots m i_j m \cdots m i_l}
+ (l-1) p_{m m \cdots mm i_l} .
\end{equation}
This linear form is non-zero and has the underlined leading term
if and only if at least two of the entries of the $l$-tuple
$(i_0,i_1,\ldots,i_{l-1})$ are not equal to $m$.
Thus the number of distinct Gr\"obner basis elements (\ref{eq:linrel1})
equals $\,m^{l+1} - m (1+l(m-1))$.

Our second class of Gr\"obner basis elements consists of the linear relations
\begin{equation}
\label{eq:linrel2}
\begin{matrix}
& \underline{p_{m \cdots m i_j m \cdots m 1}} + 
p_{m \cdots m i_j m \cdots m 2} + \cdots + 
p_{m \cdots m i_j m \cdots m m}  \\
- & 
p_{m \cdots m m m \cdots m 1} -
p_{m \cdots m m m \cdots m 2} - \cdots -
p_{m \cdots m m m \cdots m m} .
\end{matrix}
\end{equation}
These linear forms are non-zero with the underlined
leading term provided $0 \leq j \leq l-1$ and $1 \leq i_j \leq m-1$.
The number of distinct linear forms (\ref{eq:linrel2}) equals $l(m-1)$, and the set of
their leading terms is disjoint  from the set of leading terms in (\ref{eq:linrel1}).

The number of unknowns $p_\bullet$ not yet underlined equals
$\,l(m-1)^2+(m-1) + 1 $. We use these unknowns
to form $m-1$ matrices $A_2,A_3,\ldots,A_{m}$, each having format $l \times (m-1)$,
as follows. Define the matrix $A_r$ by placing the following entry in
row $j$ and column $i_j$:
\begin{equation}
\label{eq:matrixrel1}
\underline{p_{m \cdots m i_j m \cdots m r}} \,-\, p_{m \cdots mmm \cdots m r} .
\end{equation}
We finally form an $l \times (m-1)^2$ matrix by concatenating these $m-1$ matrices:
\begin{equation}
\label{eq:matrixrel2} A \,\, = \,\, \bigl( \,A_2 \, A_3 \, \,\cdots \,\,A_{m} \bigr) .
\end{equation}
The third and last group of polynomials in our Gr\"obner basis 
is the set of $2 \times 2$-minors of $A$. The entries of $A$ have
distinct leading terms, underlined in (\ref{eq:matrixrel1}), and the
leading term of each $2 \times 2$-minor is the product of the
leading terms on the main diagonal.

Note that we could also define the matrix $A_1$ and include it when forming
(\ref{eq:matrixrel2}). This would not change the ideal, but it would lead to
a generating set that is not minimal.

It is well-known that the $ 2 \times 2$-minors of a matrix
of unknowns form a Gr\"obner basis for the prime ideal they generate.
Since no unknown $p_\bullet$ underlined in (\ref{eq:linrel1}) or (\ref{eq:linrel2})
appears in the matrix $A$, it follows that these linear relations
together with the $2 \times 2$-minors of (\ref{eq:matrixrel2})
generate a prime ideal and form a Gr\"obner basis for that prime ideal.

The ideal of $2 \times 2$ minors of $A$ has codimension
$l(m-1)^2 - l - (m-1)^2 + 1$. Subtracting this quantity from
the number $\,l(m-1)^2+(m-1) + 1 \,$ of unknowns not underlined
in (\ref{eq:linrel1}) or (\ref{eq:linrel2}), we obtain
$\, l+(m-1)^2-1 + (m-1) + 1 \, = \,(m-1)m+l$.
This is the dimension of the affine variety defined by
our prime ideal. The corresponding irreducible projective variety 
has dimension $\,(m-1)m+l-1$. This is
precisely the dimension of $\overline{{\rm MTD}}_{l,m}$.

It hence suffices to prove that our variety contains the
model ${\rm MTD}_{l,m}$, or, equivalently, that
the linear forms (\ref{eq:linrel1}) and (\ref{eq:linrel2}) 
are mapped to $0$ by the parameterization (\ref{eq:param}),
and that the specialized matrix $\phi_{l,m}(A)$ has rank $1$.
For (\ref{eq:linrel2}) this is obvious because, for fixed $i_j$, 
$$ \sum_{r=1}^m \phi^*_{l,m} \bigl( p_{m \cdots m i_j m \cdots m r} \bigr)\,\, = \,\,
 \frac{1}{m^l} .$$
 Here $\phi^*_{l,m}$ denotes the homomorphism of polynomial rings
 induced by the map $\phi_{l,m}$.
 
 The indices of the unknowns in the linear form (\ref{eq:linrel1}) all have the same letter $i_l$
 in the end.  The formula (\ref{eq:param}) for the corresponding
 probabilities can thus be written as
$$ \phi^*_{l,m} (p_{i_0 i_1 \cdots i_{i-1} i_l})  \,\, = \,\,
u + x_{i_0} + y_{i_1} + \cdots + z_{i_{l-1}}. $$
In other words, for any fixed $i_l$, the resulting
$l$-dimensional tensor  has tropical rank $1$.
This representation implies linear relations like
(\ref{eq:lin2}), and these are equivalent to (\ref{eq:linrel1}).

Finally, if we apply our ring homomorphism to (\ref{eq:matrixrel1}) then we get
\begin{equation}
\label{eq:getlambda}
 \phi_{l,m}^*(p_{m \cdots m i_j m \cdots m r}) \,-\, \phi_{l,m}^*(p_{m \cdots mmm \cdots m r})
\,\, = \,\,\lambda_{j} \cdot (q_{i_j,r} -q_{m,r}). 
\end{equation}
Thus, the matrix $\phi_{k,l}(A)$ is the product of the column vector
$(\lambda_1,\ldots,\lambda_l)$ and a row vector
of length $(m-1)^2$ whose entries are
$q_{i_j,r}-q_{m,r}$ for $2 \leq r \leq m$ and $1 \leq i_j \leq m-1$.
In particular, the matrix $\phi_{l,m}^*(A)$ has rank $\leq 1$.
This completes the proof of Theorem~\ref{thm:main}.
\end{proof}

\begin{remark} \rm
The prime ideal in Theorem \ref{thm:main} is the kernel of $\phi_{l,m}^*$, so it
 characterizes the image
of the model parametrization $\phi_{l,m}$. On the model ${\rm MTD}_{l,m}$, the map $\phi_{l,m}$ can be
inverted as long as the rows of the transition matrix $(q_{ij})$ are distinct.
Indeed, $q_{ij} $ equals $2^l  \phi^*_{l,m}(p_{ii\cdots iij})$, and
the coordinates of $\lambda$ are identified from (\ref{eq:getlambda}).
Thus, our result refines the well-known
fact that MTD models are identifiable \cite[\S 4.2]{BR}.
\end{remark}

\begin{example} \label{ex:neun} \rm
We illustrate Theorem \ref{thm:main} for the case $l=2, m=3$, 
by presenting the Gr\"obner basis promised in Example~\ref{ex23}.
Note that $N = 26$.
Here the ambient linear space has dimension
$(m-1)(lm-l+1)= 10$, and our Gr\"obner basis for that linear space
consists of twelve linear forms (\ref{eq:linrel1}) 
and four linear forms (\ref{eq:linrel2}).
These are respectively,
$$ \begin{matrix}
\underline{p_{111}} {-} p_{311} {-} p_{131} {+} p_{331} ,\,
\underline{p_{121}} {-} p_{321} {-} p_{131} {+} p_{331} ,\,
\underline{p_{211}} {-} p_{311} {-} p_{231} {+} p_{331} , \,
\underline{p_{221}} {-} p_{321} {-} p_{231} {+} p_{331} ,\\
\underline{p_{112}} {-} p_{312} {-} p_{132} {+} p_{332} , \,
\underline{p_{122}} {-} p_{322} {-} p_{132} {+} p_{332} ,\,
\underline{p_{212}} {-} p_{312} {-} p_{232} {+} p_{332} , \,
\underline{p_{222}} {-} p_{322} {-} p_{232} {+} p_{332} ,\\
\underline{p_{113}}{-} p_{313} {-} p_{133} {+} p_{333} ,  \,
\underline{p_{123}} {-} p_{323} {-} p_{133} {+} p_{333} ,\,
\underline{p_{213}} {-} p_{313} {-} p_{233} {+} p_{333} , \,
\underline{p_{223}} {-} p_{323} {-} p_{233} {+} p_{333}.
\end{matrix}
$$
$$ \begin{matrix} {\rm and} \quad &
\underline{p_{311}} +p_{312}+p_{313}-p_{331}-p_{332}-p_{333}\,,\,\,
\underline{p_{321}} +p_{322}+p_{323}-p_{331}-p_{332}-p_{333} \,,\\ &
\underline{p_{131}} +p_{132}+p_{133}-p_{331}-p_{332}-p_{333} \,,\,\, 
\underline{p_{231}} +p_{232}+p_{233}-p_{331}-p_{332}-p_{333}.
\end{matrix}
$$
The remaining $l(m-1)^2 + (m-1) + 1 =8 + 2 + 1 =11$ not yet underlined 
unknowns are
$\,p_{132}, p_{232}, p_{312}, p_{322},
p_{133}, p_{233}, p_{313}, p_{323}, \,
p_{332}, p_{333},\,p_{331}$.
These represent coordinates on  the linear subspace $\PP^{10}$ of $\PP^{26}$ 
that is cut out by these linear forms. Inside that linear subspace $\PP^{10}$, our
variety $\overline{{\rm MTD}}_{2,3} $ has codimension $3$, and it is
defined ideal-theoretically by the $2 \times 2$-minors of the $2 \times 4$-matrix
$$ A \quad = \quad \bigl(\,A_2 \,\,A_3 \,\bigr) \,\, = \,\,
\begin{pmatrix}
\, \underline{p_{132}}-p_{332} & \underline{p_{232}}-p_{332} \,
&\, \, \underline{p_{133}}-p_{333} & \underline{p_{233}}-p_{333} \,\\
\, \underline{p_{312}}-p_{332} & \underline{p_{322}}-p_{332} \,&\, \, 
\underline{p_{313}}-p_{333} & \underline{p_{323}}-p_{333}\,
 \end{pmatrix}.
 $$
These six quadrics, together with the
 $16$ linear forms, form a reduced Gr\"obner basis. \qed
\end{example}

Our proof of   Theorem \ref{thm:main}
gives rise to the following geometric description:

\begin{corollary} \label{cor:segre}
The projective variety $\overline{{\rm MTD}}_{l,m} $ 
is a cone with base $\PP^{m-1}$ over the Segre variety
$\PP^{l-1} \times \PP^{m^2-2m}$. If $m \geq 3$, then
this variety is singular and its singular locus is the $\PP^{m-1}$
that forms the base of that cone.
 The degree of $\overline{{\rm MTD}}_{l,m} $ equals $\binom{l+(m-1)^2-2}{l-1}$.
\end{corollary}

\begin{proof}
The ideal of singular locus of 
$\overline{{\rm MTD}}_{l,m} $ is generated by
the entries of the matrix $A$ together with the
linear forms (\ref{eq:linrel1}) and (\ref{eq:linrel2}).
Together, these linear equations are equivalent to
requiring that the value of $\,p_{i_0 i_1 \cdots i_{l-1} r} \,$
depends only on $r$. It does not on
$i_0 i_1 \cdots i_{l-1}$. These constraints
define a linear space $\PP^{m-1}$ in $\PP^N$.
The $2 \times 2$-minors of an $l \times (m{-}1)^2$ matrix
define the Segre variety $\PP^{l-1} \times \PP^{m^2-2m}$,
 whose degree is known 
to be the  binomial coefficient.
\end{proof}


\begin{thebibliography}{10}

\bibitem{Ber} A.~Berchtold:
Estimation in the mixture transition distribution model,
{\em J. Time Ser. Anal.} {\bf 22} (2001) 379--397. 


\bibitem{BR}
A.~Berchtold and A.~Raftery:
The mixture transition distribution model for high-order
Markov chains and non-Gaussian time series,
{\em Statistical Science} {\bf 17} (2002) 328--356.

\bibitem{Cri} A.~Critch:
Binary hidden Markov models and varieties,
{\tt arXiv:1206.0500}.

\bibitem{LAS} M.~Drton, B.~Sturmfels and S.~Sullivant:
{\em Lectures on Algebraic Statistics}, Oberwolfach Seminars {\bf 39},
Birkh\"auser Verlag, Basel, 2009.


\bibitem{HT} 
H.~Hara and A.~Takemura: A Markov basis for two-state toric homogeneous Markov chain model without initial parameters, {\em J. Japan Statist. Soc.} {\bf 41} (2011) 33--49.

\bibitem{HDY}
D.~Haws, A.~Martin Del Campo and R.~Yoshida:
Degree bounds for a minimal Markov basis for the three-state toric homogeneous Markov chain model,
in T.~Hibi: {\em Harmony of Gr\"obner Bases and the Modern Industrial Society},  2012, pp.~63--98.

\bibitem{HKS}
S.~Ho\c sten, A.~Khetan and B.~Sturmfels:
Solving the likelihood equations,
{\em Foundations of Computational Mathematics} {\bf 5} (2005) 389--407.

\bibitem{ascb}
L.~Pachter and B.~Sturmfels:
{\em Algebraic Statistics for Computational Biology},
Cambridge University Press, 2005.

\bibitem{Raf} A.~Raftery:
A model for high-order Markov chains,
{\em J. Roy. Statist. Soc. Ser. B} {\bf 47} (1985) 528--539. 

\bibitem{RT} A.~Raftery and S.~Taver\'e: Estimation and modelling
repeated patterns in high order Markov chains with the Mixture Transition
Distribution Model, {\em Applied Statistics} (1994) 179--199.

\bibitem{Ric}
D.~Richards:
Counting and locating the solutions of polynomial systems
of ML equations, presentation at the International
Workshop in Applied Probability, University of Connecticut, 2006.

\bibitem{Sch} A.~Sch\"onhuth: 
Generic identification of binary-valued hidden Markov processes,
{\tt  arXiv:1101.3712}.

\bibitem{Zie} G.~Ziegler:
{\em Lectures on Polytopes},
Graduate Texts in Mathematics, 152,
Springer-Verlag, New York, 1995.

\end{thebibliography}
\end{document}